\documentclass{article}
\usepackage{graphicx} % Required for inserting images
\usepackage{hyperref} % Load hyperref after natbib for proper link handling
\usepackage{amsmath, amsthm, amssymb}
\usepackage{tikz}
\usepackage{tcolorbox}
\usepackage{dashrule} % for dashed lines
\usepackage{xcolor}
\usepackage{wrapfig}
\usepackage{fullpage}

\makeatletter
\newcommand*{\rom}[1]{\expandafter\@slowromancap\romannumeral #1@}
\makeatother

\usepackage[ruled, lined, linesnumbered, longend]{algorithm2e}

\newtheorem{theorem}{Theorem}

\newtheorem{lemma}{Lemma}

\title{Dvo{\v{r}}{\'a}k--Dell--Grohe--Rattan theorem via an asymptotic argument}

\author{ Alexander Kozachinskiy\footnote{alexander.kozachinskyi@cenia.cl. Supported by the ANID Fondecyt Iniciación grant 11250060 and National
Center for Artificial Intelligence CENIA FB210017, Basal ANID.} }
\date{Centro Nacional de Inteligencia Artificial, Chile
}

\begin{document}

\maketitle
\begin{abstract}
   Two graphs $G_1,G_2$ are distinguished by the Weisfeiler--Leman isomorphism test if and only if there is a tree $T$ that has a different number of homomorphisms to $G_1$ and to $G_2$. There are two known proofs of this fact -- a logical proof by Dvo{\v{r}}{\'a}k and a linear-algebraic proof by  Dell, Grohe, and Rattan.   We give another simple proof, based on ordering WL-labels and asymptotic arguments.
\end{abstract}

The Weisfeiler--Leman (WL) test~\cite{weisfeiler1968reduction} is an algorithm meant for checking if two given graphs $G_1$ and $G_2$ are isomorphic. It works as follows. First, one defines a notion of the \emph{$k$-level} WL-label of a node of a graph, inductively over $k$. For $k = 0$, all nodes in all graphs have the same $0$-level WL-label. Now, the $(k+1)$-level WL-label of a node is the multiset of $k$-level WL-labels of its neighbors. Two graphs $G_1$ and $G_2$ are distinguished by the WL-test if for some $k$, they have different multisets of $k$-level WL-labels of their nodes.

This is a sound isomorphism test (isomorphic graphs cannot be distinguished) but not complete -- for instance, a hexagon will not be distinguished from two disjoint triangles. It is known that two graphs $G_1$ and $G_2$ are distinguished by the WL test if and only if there exists a tree $T$  such that there is a different number of homomorphisms from $T$ to $G_1$ and from $T$ to $G_2$. Here by a homomorphism from a graph $H$ to a graph $G$ we mean a mapping from the set of nodes of $H$ to the set of nodes of $G$ such that adjacent nodes in $H$ are mapped to adjacent nodes in $G$. This result was first proved by Dvo{\v{r}}{\'a}k~\cite{dvovrak2010recognizing} using a logical argument, and then re-discovered by Dell, Grohe, and Rattan~\cite{dell2018lovasz} with a linear-algebraic argument.

There is a simple direction of this theorem -- if there is a tree $T$ that has a different number of homomorphisms to $G_1$ and $G_2$, then $G_1$ and $G_2$ are distinguished by the WL test. This follows from an observation that the $k$-level WL-label $\ell_k(u)$ of a node $u$ in a graph $G$ uniquely determines the number $h(T,\ell_k(u))$ of homomorphisms from a depth-$k$ tree $T$ with an indicated ``root'' to $G$ such that the root goes to $u$. Hence, for a $k$-depth tree $T$, if we arbitrarily choose one of its nodes as a root, the number of homomorphisms from $T$ to a graph $G = (V, E)$ is:
\[\sum\limits_{u\in V} h(T,\ell_k(u)).\]
If for some depth-$k$ tree $T$ and for two graphs $G_1 = (V_1, E_1), G_2 = (V_2, E_2)$, we have:
\[\sum\limits_{u_1\in V_1} h(T,\ell_k(u)) \neq \sum\limits_{u_2\in V_2} h(T,\ell_k(u_2)),\]
then the multisets of $k$-level WL-labels of the nodes of $G_1$ and $G_2$ have to be different.

We now establish the hard direction via a new argument.

\begin{theorem}
    If two graphs $G_1$ and $G_2$ are distinguished by the WL test, then there is a tree $T$ that has a different number of homomorphisms to $G_1$ and to $G_2$.
\end{theorem}
\begin{proof}

    For the proof, we define, for any $k$, a linear order on $k$-level WL-labels. For $k = 0$, there is just one possible label, so there is nothing to define. For $k = 1$, each label is identified with a natural number (degree of a node in a graph). The larger this number is, the larger we put the label in our order.

    More generally, assume we have already defined the order on $k$-level labels. We extend it to $(k+1)$-level labels ``lexicographically''. Namely, $(k+1)$-level labels are multisets of $k$-level labels. Now, we take two $(k+1)$-labels $\ell_1$ and $\ell_2$ that we want to ``compare''. There is already an order on their elements (these elements are $k$-level labels). We take the largest element that appears a different number of times in the multisets $\ell_1$ and $\ell_2$. The multiset that has this element more times is defined to be the larger one. This is clearly a linear order.

\begin{lemma}
\label{lem_dom}
    For any finite set $S$ of $WL$-labels that does not have labels of isolated nodes, and for any $k$, there exists  a sequence of $k$-depth rooted trees $\{T^k_n\}_{n\to\infty}$ such that for any two $k$-level labels $\ell_1, \ell_2\in S$, we have:
     \begin{equation}
     \label{eq_order}
         \ell_1 < \ell_2 \implies h(T^k_n, \ell_1) = o(h(T^k_n, \ell_2)) \text{ as } n\to\infty.\
     \end{equation}
\end{lemma}

%\begin{remark}
 %   Since $S$ does not have labels of isolated nodes, we have $h(T, \ell) \ge 1$ for any tree $T$ and for any $\ell \in L$. Hence, the conditions  $h(T^k_n, \ell_1) = o(h(T^k_n, \ell_2)) \text{ as } n\to\infty$ are well-defined, because the fraction is $h(T^k_n, \ell_1)/h(T^k_n, \ell_2)$ is well-defined.
%\end{remark}

    Let us first deduce the theorem from the lemma. Take any two graphs $G_1$ and $G_2$ that are distinguished after the $k$-th iteration of the WL-test.  Assume first, for simplicity, that $G_1$ and $G_2$ do not have isolated nodes.
    We apply the lemma to the set $L$of $k$-level WL labels that appear in $G_1$ and $G_2$. We will show that one of the trees $T^k_n$ distinguishes $G_1$ and $G_2$ by the number of homomorphisms from it (as from a non-rooted tree). Indeed, the number of homomorphisms from $T^k_n$ to a graph $G$ can be counted as the sum over vertices $v$ of $G$, of the number of homomorphisms from $T^k_n$ to $G$ that map the root to $v$. Hence, it is enough to show that for a large enough $n$, we have:
    \begin{equation}
    \label{eq_dif}
    \sum\limits_{v\in V(G_1)} h(T^k_n,\ell^k(v)) \neq \sum\limits_{v\in V(G_2)} h(T^k_n,\ell^k(v)). 
    \end{equation}
    Graphs $G_1$ and $G_2$ have different multisets of $k$-level labels. Consider the largest (in our ``lexicographic'' order) $k$-level label $\ell$ that appears a different number of times in them. Let us say, it appears more in $G_1$. If we consider the difference between the left-hand side and the right-hand side in \eqref{eq_dif}, all terms corresponding to labels that are larger than $\ell$ will cancel out. In turn, $h(T_n^k, \ell)$ will be with a positive coefficient in this difference, and $h(T_n^k,\ell)$ is itself strictly positive because $\ell$ is for a non-isolated node. Hence, for large enough $n$, the difference will be strictly positive. Indeed, terms $h(T_n^k, \ell_1)$ for $\ell_1 < \ell$ can have negative coefficients, but because these terms are dominated by  $h(T_n^k, \ell)$, this will not prevent the difference from being positive in the limit.

    What if $G_1$ and $G_2$ have isolated nodes? Then in \eqref{eq_dif}, some terms are 0s, corresponding to isolated nodes. If the multisets of non-isolated nodes are different, the same argument as before applies. If they are the same, then the graphs have to differ just in the number of nodes, in which case we can distinguish them by the number of homomorphisms from a single isolated node.

    \begin{proof}[Proof of Lemma \ref{lem_dom}]
        Let us first establish the lemma for $k = 1$. We will let $T^1_n$ be just a star with $n$ children. There are $d^n$ homomorphisms from $T^1_n$ to a node of degree $d$ that map the root to this node (for every child of the root, there are $d$ options). Since $d_1^n = o(d_2^n)$ for $1\le d_1 < d_2$, the induction base follows.

We now perform the induction step. Assume that the statement is proved for $k$. First, let us for every $n$ define a tree $H_n$ of depth $k + 1$  where the root has one child whose subtree is $T_n^k$. Let $L$ be a $(k + 1)$-level WL-label, i.e., a multiset of some $k$-level WL-labels $\ell$. We have:
\[h(H_n, L) =  \sum\limits_{\ell\in L} h(T_n^k, \ell).\]
Indeed, if $v$ is a node with label $L$, then we have to map the unique child of the root of $H_n$ to some neighbor of $v$. Upon choosing this neighbor, we have to map $T_n^k$ to this neighbor, and the number of ways to do that is $h(T_n^k, \ell)$, where $\ell$ is the $k$-level WL-label of this neighbor. And by definition, the multiset of $k$-level WL-labels of neighbors of $v$ is $L$.

 Now, there exists $m$ such that $h(H_m, L_1) < h(H_m, L_2)$ for all $L_1 < L_2 \in S$. Indeed, based on the condition \eqref{eq_order} for level $k$, we have for any two fixed $L_1, L_2 \in S$ with $L_1 < L_2$ 
 \[h(H_n, L_1) < h(H_n, L_2)\]
 for all large enough $n$ (we consider the largest $k$-level label $\ell$ that appears a different number of times in $L_1$ and $L_2$, there are fewer of them in $L_1$, so the difference  $h(H_n, L_2) - h(H_n, L_1)$ includes this term with a positive coefficient, this term is itself positive because of the assumption about not having isolated nodes, and all other terms are dominated by it). There are finitely many labels in $S$, so it remains to take $m$ large enough so that this inequality holds for all pairs in question.

 What remains to do is to define $T_n^{(k+1)}$ as a depth-$(k+1)$ tree where the root has $n$ children, each followed by $T_m^k$ (or, in other words, $n$ copies of $H_m$, glued by the root). Observe that 
 \[h(T_n^{k+1}, L) = h(H_m, L)^n.\]
Taking any two $L_1, L_2 \in S$ with $L_1 < L_2$, since  $h(H_m, L_1) < h(H_m, L_2)$, we obtain:
\[h(T_n^{k+1}, L_1) = h(H_m, L_1)^n = o( h(H_m, L_2)^n) = o(h(T_n^{k+1}, L_2)),\]
as required.
    \end{proof}
\end{proof}


\begin{thebibliography}{1}

\bibitem{dell2018lovasz}
{\sc Dell, H., Grohe, M., and Rattan, G.}
\newblock Lov{\'a}sz meets weisfeiler and leman.
\newblock In {\em 45th International Colloquium on Automata, Languages, and
  Programming (ICALP 2018)\/} (2018), Schloss Dagstuhl--Leibniz-Zentrum f{\"u}r
  Informatik, pp.~40--1.

\bibitem{dvovrak2010recognizing}
{\sc Dvo{\v{r}}{\'a}k, Z.}
\newblock On recognizing graphs by numbers of homomorphisms.
\newblock {\em Journal of Graph Theory 64}, 4 (2010), 330--342.

\bibitem{weisfeiler1968reduction}
{\sc Weisfeiler, B., and Leman, A.}
\newblock The reduction of a graph to canonical form and the algebra which
  appears therein.
\newblock {\em nti, Series 2}, 9 (1968), 12--16.

\end{thebibliography}
\end{document}